\documentclass[11pt]{amsart}
\usepackage{amssymb}
\usepackage{amsthm}

\newtheorem{theorem}{Theorem}[section]

\newtheorem{corollary}[theorem]{Corollary}
\newtheorem{lemma}[theorem]{Lemma}
\newtheorem{proposition}[theorem]{Proposition}
\theoremstyle{definition}
\newtheorem{definition}[theorem]{Definition}

\theoremstyle{remark}

\numberwithin{equation}{section}

\def\R {{\mathbb{R}}}

\def\Z {{\mathbb{Z}}}

\def\3{{|\!|\!|}}

\begin{document}
\title[On equivalence relations induced by TSI Polish groups]{On equivalence relations induced by locally compact TSI Polish groups admitting open identity component}
\author{Yang Zheng}
\address{Academy of Mathematics and Systems Science, Chinese Academy of Sciences, East
Zhong Guan Cun Road No. 55, Beijing 100190, China}
\email{yangz@amss.ac.cn }


\begin{abstract}

For a Polish group $G$, let $E(G)$ be the right coset equivalence relation $G^\omega/c(G)$, where $c(G)$ is the group of all convergent sequences in $G$.

We prove a Rigid theorem on locally compact TSI Polish groups admitting open identity component, as follows: Let $G$ be a locally compact TSI Polish group such that $G_0$ is open in $G$, and let $H$ be a nontrivial pro-Lie TSI Polish group. Then $E(G)\leq_BE(H)$ iff there exists a continuous homomorphism $\phi:G_0\to H_0$ satisfying the following conditions:
\begin{enumerate}
  \item [(i)] $\ker(\phi)$ is non-archimedean;
  \item [(ii)] $\phi{\rm Inn}_G(G_0)\subseteq\overline{{\rm Inn}_H(H_0)\phi}$ under pointwise convergence topology.
\end{enumerate}

 An application of the Rigid theorem yields a negative answer to Question 7.5 of~\cite{DZ}.
\end{abstract}

\maketitle
\section{Introduction}
Equivalence relations appear in various branches of mathematic.
In order to measure the relative complexity of equivalence relations, logicians introduce the concept of Borel reducibility.
A topological space is {\it Polish} if it is separable and completely metrizable. Given two Borel equivalence relations $E$ and $F$ on Polish spaces $X$ and $Y$ respectively, recall that $E$ is {\it Borel reducible} to $F$, denoted $E\leq_B F$, if there exists a Borel map $\theta:X\rightarrow Y$ such that for all $x,y\in X$,
$$xEy\Longleftrightarrow\theta(x)F\theta(y).$$
We denote $E\sim_B F$ if both $E\leq_B F$ and $F\leq_B E$, and denote $E<_B F$ if $E\leq_B F$ and $F\nleq_B E$. For more details on Borel reducibility, we refer to~\cite{gaobook}.

Polish groups are important tools in the research on Borel reducibility. We say that a topological group is {\it Polish} if its topology is Polish. All countable discrete groups or compact metrizable groups are Polish.

For a Polish group $G$, the concept of so called $E(G)$ equivalence relation has been introduced in~\cite{DZ}, which provides us with a new perspective on understanding the structure of Polish groups.
The equivalence relation $E(G)$ on $G^\omega$ is defined by
$$xE(G)y\iff\lim_nx(n)y(n)^{-1}\mbox{ converges in }G$$
for $x,y\in G^\omega$. We say that $E(G)$ is the {\it equivalence relation induced by} $G$.

A series of results regarding $E(G)$ have been discovered in~[2-4], such as some Rigid theorems on connected Polish groups. For example: Let $G$ be a compact connected abelian Polish group and $H$ a locally compact abelian Polish group. Then $E(G)\leq_B E(H)$ iff there is a continuous homomorphism $S:G\rightarrow H$ such that $\ker(S)$ is non-archimedean~(cf.~\cite[Theorem 1.2]{DZlco}). The main purpose of this article is to prove a Rigid theorem on disconnected Polish groups.

A {\it Lie} group is a group which is also a smooth manifold such
that the group operations are smooth functions. A completely metrizable topological
group $G$ is called a {\it pro-Lie} group if every open neighborhood of $1_G$ contains
a closed normal subgroup $N$ such that $G/N$ is a Lie group~(cf.~\cite[Definition 1]{HM07}).
Every locally compact TSI Polish group is pro-Lie~(cf.~\cite[Theorem 3.6]{HMS}). In fact, Using the definition of pro-Lie groups, ~\cite[Theorem 3.34]{HM07} and~\cite[Theorem 3.35]{HM07}, one show that a TSI Polish group is pro-Lie iff it is topologically isomorphic to a closed subgroup of a countable product separable TSI Lie groups.

Let $G$ be a Polish group and let $G_c$ be a closed normal subgroup of $G$.
For $u\in G$, we define a topological group automorphism $\iota_u:G_c\to G_c$ by $\iota_u(g)=ugu^{-1}$, and denote
$${\rm Inn}_G(G_c)=\{\iota_u:u\in G\}.$$

Given an abstract group $G$, the identity element of $G$ is denoted by $1_G$ . If $G$ is a
topological group, the connected component of $1_G$ in $G$~(identity component) is denoted by $G_0$.
The identity component of a Lie group $G$ is open and normal in $G$. Recall that a Polish group $G$ is {\it non-archimedean} if it has a neighborhood
base of {\it $1_G$} consisting of open subgroups.

We establish the main theorem of this article.

\begin{theorem}[Rigid theorem on locally compact TSI Polish groups admitting open identity component]\label{2}
Let $G$ be a locally compact TSI Polish group such that $G_0$ is open in $G$, and let $H$ be a  nontrivial pro-Lie TSI Polish group. Then $E(G)\leq_BE(H)$ iff there exists a continuous homomorphism $\phi:G_0\to H_0$ satisfying the following conditions:
\begin{enumerate}
  \item [(i)] $\ker(\phi)$ is non-archimedean;
  \item [(ii)] $\phi{\rm Inn}_G(G_0)\subseteq\overline{{\rm Inn}_H(H_0)\phi}$ under pointwise convergence topology, i.e., for any $u\in G$, there is some $(h_n)\in H^\omega$ satisfying $\phi(ugu^{-1})=\lim_nh_n\phi(g)h_n^{-1}$ for all $g\in G_0$.
\end{enumerate}
\end{theorem}

For the case of $G_0\cong\R^n$, we have a positive answer to Question 7.5 of~\cite{DZ}.
Let $G$ and $H$ be two separable TSI Lie groups such that $G_0\cong H_0\cong \R^n$ and $|{\rm Inn}_H(H_0)|<+\infty$. Then $E(G)\leq_BE(H)$ iff there exists a topological isomorphism $\phi:G_0\rightarrow H_0$ such that $\phi{\rm Inn}_G(G_0)\phi^{-1}\subseteq {\rm Inn}_H(H_0)$(see Corollary~\ref{the Rn finit}). However, it is not true in general.

Let $G$ and $\Lambda$ be two groups, and let $\alpha$ be a homomorphism from $\Lambda$ to ${\rm Aut}(G)$, the group of automorphisms of $G$.
 Recall that the {\it semi product} $\Lambda\ltimes_\alpha G$ is the set $\Lambda\times G$ equipped with group operation as,
$$(\lambda_1,g_1)(\lambda_2,g_2)=(\lambda_1\lambda_2,g_1\alpha(\lambda_1)(g_2)),$$
where $(\lambda_1,g_1),(\lambda_2,g_2)\in\Lambda\times G$.

Note that, if $G$ is a Polish group, $\Lambda$ is a countable discrete group, and $\phi(\lambda)$ is a continuous automorphism on $G$ for each $\lambda$, then $\Lambda\ltimes_\phi G$ equipped with the product topology on $\Lambda\times G$ is also a Polish group.

Let ${\rm GL}(2,\R)$ be the group of all invertible $2\times 2$ real matrices. For any countable subgroup $\Lambda$ of ${\rm GL}(2,\R)$, we define a continuous homomorphism $\Upsilon_\Lambda:\Lambda\rightarrow {\rm Aut}(\R^2)$  with $\Upsilon_\Lambda(A)(t)=A t$ for $t\in\R^2$ and $A\in \Lambda$. Endow $\Lambda$ with discrete topology. We shall denote $\Lambda\ltimes_{\Upsilon_\Lambda} \R^n$ by $\Lambda\ltimes\R^n$ for brevity.
 The group of all rotations of the euclidean plane $\R^2$ is given by
$${\rm SO(2)}=\left\{R(t):R(t)={\normalsize\begin{pmatrix}
     \cos(2\pi t)&-\sin(2\pi t) \\
     \sin(2\pi t)&\cos(2\pi t)
   \end{pmatrix}},t\in\R\right\}.$$
Denote ${\Gamma}^z=\{R(\sqrt{2}n):n\in\Z\}$ and
$\Gamma_k=\{R(i\frac{2\pi}{k}):i\leq k\}~(k\geq 2).$

The following result provide a negative answer to Question 7.5 of~\cite{DZ}.
\begin{theorem}
Let $G=\Gamma_k\ltimes\R^2$ and $H={\Gamma}^z\ltimes\R^2$. Then $E(G)<_B E(H)$, but there is no topological isomorphism $\phi:G_0\to H_0$ such that $\phi{\rm Inn}_G(G_0)\phi^{-1}\subseteq{\rm Inn}_H(H_0)$.
\end{theorem}

This article is organized as follows. In section 2, we prove some Borel reducibility results on $E(G)$.  In section 3, we prove Theorems 1.1, 1.2.

\section{Preliminaries}

Let $G$ be a Polish group. the authors~\cite{DZ} defined an equivalence relation $E_*(G)$ on $G^\omega$ as: for $x,y \in G^\omega$,
$$x E_*(G)y\iff\lim_nx(0)x(1)\dots x(n)y(n)^{-1}\dots y(1)^{-1}y(0)^{-1}\mbox{ converges}.$$

Suppose $G$ is TSI. It is more convenient to take $E_*(G)$
as research object than $E(G)$. It is trivial that $E(G)\sim_B E_*(G)$ (cf.~\cite[Proposition 2.2]{DZlco}). For brevity, we define
$$(x,y)|_m^n=x(m)\cdots x(n)y(n)^{-1}\cdots y(m)^{-1}$$
for $x,y\in G^\omega$ and $m\leq n$. Let $d$ be a complete compatible two-side invariant metric on $G$.
It is clear that $xE_*(G)y\iff \lim_m\sup_n d(1_G,(x,y)|_m^n)=0$.

Special equivalence relations $E(G;0)$ can be found in~\cite[6.1]{DZ}, which is defined on $G^\omega$ as
$xE(G;0)y\iff \lim_nx(n)^{-1}y(n)=1_G.$

Let $G,H$ be Polish groups, and let $d_G,d_H$ be left invariant compatible metrics on $G$ and $H$ respectively. Then we denote by ${\rm Hom}(G,H)$ the space of all continuous homomorphisms of $G$ to $H$.
\begin{definition}
 For $\Omega,\Omega'\subseteq{\rm Hom}(G,H)$ and identity neighborhood $U$ of $G$, we say that $\Omega$ is {\it locally approximated by $\Omega'$ on $U$ }if the following condition is satisfied:
$$\forall \varphi\in \Omega\,\exists(\varphi_n)\in (\Omega')^\omega\,(\lim_n\sup\{d_H(\varphi(g),\varphi_n(g)):g\in U\}=0).$$
\end{definition}
It is easy to check that the above definition does not depend on the choice of left invariant compatible metrics $d_G$ and $d_H$. We say that $\Omega$ is {\it equicontinuous}, if for any open identity neighborhood $W$ of $H$, there is an open identity neighborhood $V$ of $G$ such that $\varphi(V)\subseteq W$ for all $\varphi\in \Omega$.

\begin{proposition}\label{pointwis}
Let $G$ and $H$ be two Polish groups such that $G$ is locally compact and let $\Omega,\Omega'\subseteq {\rm Hom}(G,H)$. Suppose $\Omega'$ is equicontinuous and $\Omega\subseteq\overline{\Omega'}$ under pointwise convergence topology, then
 $\Omega$ is locally approximated by $\Omega'$ on an open identity neighborhood $U$.
\end{proposition}
\begin{proof}
Let $d_G,d_H$ be complete compatible two-side invariant metrics on $G$ and $H$ respectively.
Since $G$ is locally compact, we can find an open identity neighborhood $U$ so that $\overline{U}$ is compact. Fix a $\varphi\in\Omega$, then there is $(\varphi_n)\in (\Omega')^\omega$ with $\lim_n\varphi_n(g)=\varphi(g)$ for all $g\in G$. Let $\varepsilon>0$. Note that $\varphi$ is continuous. Again since $\Omega'$ is equicontinuous, we can find a $\delta>0$ with $$\sup\{d_H(\varphi_n(g),1_H),d_H(\varphi(g),1_H): d_G(g, 1_G)<\delta,n\in\omega\}<3^{-1}\varepsilon.$$

  By compactness, there are finitely many points $g_0,g_1,\cdots,g_m\subseteq U$ such that, for all $g\in U$, there is $i\leq m$ with $d_G(g_ig^{-1},1_G)<\delta$. Then there is a $N\in\omega$ with $d_{H}(\varphi_n(g_i),\varphi(g_i))<3^{-1}\varepsilon$ for all $i\leq m$ and $n\geq N$. Next fix a $g\in U$, we find a $g_{i}$ such that $d_G(g_{i}{g}^{-1},1_G)<\delta$. Then for each $n\geq N$, we have
$$\begin{aligned}
d_H(\varphi_n(g),\varphi(g))&\leq d_H(\varphi_n(g_{i}),\varphi_n(g))+d_H(\varphi_n(g_{i}),\varphi(g_{i}))+d_H(\varphi(g_{i}),\varphi(g))\cr
    &\leq \varepsilon.
  \end{aligned}$$
  This show that $\lim_n\sup\{d_H(\varphi(g),\varphi_n(g)):g\in U\}=0$. Therefore, $\Omega$ is locally approximated by $\Omega'$ on $U$.
\end{proof}

Now we generalize Theorem 6.7 of~\cite{DZ}.

\begin{lemma}\label{the borl red}
Let $G$ and $H$ be two TSI Polish groups, $G_c$ and $H_c$ two closed normal subgroups of $G$ and $H$ respectively. Suppose $G_c$ is also open in $G$. If there exist a continuous homomorphism $\phi:G_c\to H_c$ and an open subset $U\ni1_G$ of $G_c$ satisfying the following conditions:
\begin{enumerate}
  \item [(i)] $\ker(\phi)$ is non-archimedean;
  \item [(ii)] $\phi{\rm Inn}_G(G_c)$ is locally approximated by ${\rm Inn}_H(H_c)\phi$ on $U$;
  \item [(iii)]

  let $\widetilde{\phi}:G_c/\ker(\phi)\rightarrow H_c$ with $\widetilde{\phi}(g\ker(\phi))=\phi(g)$ for $g\in G_c$ and $U\ker(\phi)=\{g\ker(\phi):g\in U\}$, then $\widetilde{\phi}|_{U\ker(\phi)}$ is a homeomorphism from $U\ker(\phi)$ onto $\phi(U)$;
  \item [(\romannumeral 4)]for any $g\in G$, $\iota_g(\ker(\phi))=\ker(\phi)$,
\end{enumerate}
then $E(G)\le_BE(H)\times E(G;0)$.
\end{lemma}
\begin{proof}
Let $d_G,d_H$ be complete compatible two-side invariant metrics on $G$ and $H$ respectively.
By Corollary 3.6 of~\cite{DZ}, we only need to show that $E_*(G)\le_BE_*(H\times\mathbb{Z})\times E(G;0)$.

We fix a $T_G\subseteq G$ such that $T_G$ meets each coset of $G_c$ at exactly one point and enumerate $T_G$ as $\{w_n:n\in\omega\}$.
Put a $\delta>0$ satisfying $\{g\in G: d_G(g,1_G)<\delta\}\subseteq U$.

For each $w\in T_G,n\in\omega$, by condition (ii), we can find a $h^w_n\in H$ such that
 $$\sup\{d_H(\phi(wgw^{-1}),h^w_n\phi(g)(h^w_n)^{-1}):g\in U\}<2^{-n}.\eqno{(*)}$$
Next we define three Borel maps $\psi: G^\omega\rightarrow (H\times\mathbb{Z})^\omega\times G^\omega,\psi_1: G^\omega\rightarrow H^\omega$ and $\psi_2: G^\omega\rightarrow \mathbb{Z}^\omega$ by letting, for each $x\in G^\omega$ and $n\in\omega$,
$$\psi(x)(n)=((\psi_1(x)(n),\psi_2(x)(n)),x(n))=((h^{w_i}_n\phi(g),i),x(n)),$$
where $x(n)=w_ig$ for some $g\in G_c$. Such $w_i$ and $g$ are unique, since $|T_G\cap g^*G_c|=1$ holds for each $g^*\in G$. It remains to verify that $\psi$ is a reduction from $E_*(G)$ to $E_*(H\times\mathbb{Z})\times E(G;0)$.

Let $x,y\in G^\omega$. Since $T_G$ meets each coset of $G_c$ at exactly one point, every $x(n)$ and $y(n)$ is of the form $x(n)=u_ng_n,~y(n)=v_nh_n$ for $u_n,v_n\in T_G$ and $g_n,h_n\in G_c$.

(1) First suppose $x E_*(G) y$. Then $\lim_n x(n)y(n)^{-1}=\lim_n u_ng_nh_n^{-1}v_n^{-1}=1_G$, i.e., $x E(G;0)y$. Since $G_c$ is open, by the property of $T_G$, there is $N_1\in\omega$ such that $u_n=v_n$ for all $n>N_1$. This implies that $\psi_2(x)(n)=\psi_2(y)(n)$ for all $n>N_1$, i.e., $\psi_2(x)E_*(\mathbb{Z})\psi_2(y)$. So we only need to verify that $\psi_1(x)E_*(H)\psi_1(y)$.

We fix a $0<\varepsilon<2^{-1}\delta$, then there is a natural number $N>N_1$ so that for any $n\geq m>N$,
$$\begin{aligned}
d_G((x,y)|_m^n,1_G)=&d_G(x(m)\cdots x(n)y(n)^{-1}\cdots y(m)^{-1},1_G)\\
=&d_G(u_mg_m\cdots u_ng_nh_n^{-1}u_n^{-1}\cdots h_m^{-1}u_m^{-1},1_G)<\varepsilon.
  \end{aligned}$$
In particular, we have $d_G((x,y)|_{m}^n,1_G)=d_G(g_m(x,y)|_{m+1}^nh_m^{-1},1_G)<\varepsilon<\delta$. Put $N<k\leq l$, by the triangle inequality and condition $(*)$, we have
$$\begin{aligned}
&d_H(\phi((x,y)|_{k}^{l}),(\psi_1(x),\psi_1(y))|_{k}^l)\\
=&d_H(\phi(u_kg_k(x,y)|_{k+1}^{l}h_k^{-1}u_k^{-1}),h^{u_k}_kz_k(h^{u_k}_k)^{-1})\\
\leq & d_H(\phi(u_kg_k(x,y)|_{k+1}^{l}h_k^{-1}u_k^{-1}),h^{u_k}_k\phi(g_k(x,y)|_{k+1}^{l}h_k^{-1})(h^{u_k}_k)^{-1})+\\
&d_H(h^{u_k}_kz_k(h^{u_k}_k)^{-1},h^{u_k}_k\phi(g_k(x,y)|_{k+1}^{l}h_k^{-1})(h^{u_k}_k)^{-1})\\
\leq& 2^{-k}+d_H(z_k,\phi(g_k(x,y)|_{k+1}^{l}h_k^{-1}))\\
=& 2^{-k}+d_H((\psi_1(x),\psi_1(y))|_{k+1}^l,\phi((x,y)|_{k+1}^{l})),
\end{aligned}$$
where $z_k=\phi(g_k)(\psi_1(x),\psi_1(y))|_{k+1}^l\phi(h_k^{-1})$. It follows from $d_G(g_l h_l^{-1},1_G)<\varepsilon$ that
$$\begin{aligned}
d_H(\phi((x,y)|_{l}^{l}),(\psi_1(x),\psi_1(y))|_{l}^l)&=d_H(\phi(u_lg_lh_l^{-1}u_l^{-1}),h^{u_l}_l\phi(g_lh_l^{-1})(h^{u_l}_l)^{-1})\\
&<2^{-l}.
\end{aligned}$$
Thus $d_H(\phi((x,y)|_{k}^{l}),(\psi_1(x),\psi_1(y))|_{k}^l)<2^{-k}+\cdots 2^{-l}<2^{-k+1}$. Note that $\phi$ is continuous and $\lim_n\sup_{n\leq k\leq l}d_G((x,y)|_{k}^{l},1_G)=0$, so
$$\lim_n\sup_{n\leq k\leq l}d_H((\psi_1(x),\psi_1(y))|_{k}^{l},1_H)=0.$$
This shows that $\psi_1(x)E_*(H)\psi_1(y)$, and thus $(\psi(x),\psi(y))\in E_*(H\times\mathbb{Z})\times E(G;0)$.

(2) On the other hand, suppose $(\psi(x),\psi(y))\in E_*(H\times\mathbb{Z})\times E(G;0)$. Then there is a natural number $N_2\in\omega$ so that $\lim_n\psi_2(x)(n)=\psi_2(y)(n)$ for all $n\geq N_2$. This implies that $u_n=v_n$ for all $n\geq N_2$. Now we fix an arbitrary real number $0<\delta'<2^{-1}\delta$. Then there is a natural number $N'\geq N_2$ such that $d_G(x(n)y(n)^{-1},1_G)<2^{-1}\delta'$ for all $n\geq N'$.

Let $d_{\phi}(g\ker(\phi),g'\ker(\phi))=\inf\{d_G(gh,g'h'):h,h'\in\ker(\phi)\}$, then $d_{\phi}$ is a complete compatible two-side invariant metric on $G_c/\ker(\phi)$ (cf.~\cite[Lemma 2.2.8]{gaobook}).
Next we claim that $\lim_m\sup_{m\leq n}d_\phi((x,y)|_m^n\ker(\phi),\ker(\phi))<\delta'$. Otherwise, there are two increasing sequences of natural numbers $(m_k)$ and $(n_k)$ with $N'<m_k<n_k<m_{k+1}$ and $$d_G((x,y)|_{m_k}^{n_k},1_G)\geq d_\phi((x,y)|_{m_k}^{n_k}\ker(\phi),\ker(\phi))\geq\delta'.$$

For each $k\in\omega$, let $p_k$ be the largest natural number $p<n_k$ such that $$ \forall p'\leq p\,(d_G((x,y)|_{m_k}^{p'},1_G)<2^{-1}\delta').$$
Then we have
$$\begin{aligned}
2^{-1}\delta'\leq d_G((x,y)_{m_k}^{p_k+1},1_G)&\leq d_G(x(p_k+1),y(p_k+1))+d_G((x,y)|_{m_k}^{p_k},1_G)\\
&<2^{-1}\delta'+2^{-1}\delta'=\delta',\end{aligned}$$ and for $m_k\leq q\leq p_{k}+1$, we have
$$\begin{aligned}
d_G((x,y)|_q^{p_k+1},1_G)=&d_G((x,y)|_{m_k}^{p_k+1},(x,y)|_{m_k}^{q-1})\\
\leq&d_G((x,y)|_{m_k}^{p_k+1},1_G)+d_G((x,y)|_{m_k}^{q-1},1_G)\\
<&\delta'+2^{-1}\delta'<\delta.
\end{aligned}$$
Similar to the proof of (1), we get $$d_H(\phi((x,y)|_{m_k}^{p_k+1}),(\psi_1(x),\psi_1(y))|_{m_k}^{p_k+1})\leq 2^{-m_k+1}.$$
Note that $(\psi_1(x),\psi_1(y))\in E_*(H)$, so $\lim_k(\psi_1(x),\psi_1(y))|_{m_k}^{p_k+1}=1_H$ and hence $\lim_k\phi((x,y)|_{m_k}^{p_k+1})=1_H$. Again since $\widetilde{\phi}|_{U\ker(\phi)}:U\ker(\phi)\rightarrow\phi(U)$ is a homeomorphism and $(x,y)|_{m_k}^{p_k+1}\in U$, we get that
$\lim_k(x,y)|_{m_k}^{p_k+1}\ker(\phi)=\ker(\phi)$.
However, this contradicts $d_\phi((x,y)|_{m_k}^{n_k}\ker(\phi),\ker(\phi))\geq\delta'$. So $\lim_m\sup_{m\leq n}d_\phi((x,y)|_m^n\ker(\phi),\ker(\phi))=0$.

We define a continuous surjective homomorphism $\varphi:G_c^\omega\rightarrow (G_c/\ker(\phi))^\omega$ by letting $\varphi(z)(n)=z(n)\ker(\phi)$ for $z\in G_c^\omega$ and $n\in\omega$. Let $\hat{x},\hat{y}\in G^\omega_c$ with
$$\hat{x}(n)=u_{0}u_{1}\cdots u_{n}g_{n}u_{n}^{-1}\cdots u_{1}^{-1}u_{0}^{-1},$$
$$\hat{y}(n)=u_{0}u_{1}\cdots u_{n}h_{n}u_{n}^{-1}\cdots u_{1}^{-1}u_{0}^{-1}.$$
Then for any $N'<m\leq n$, we have
$$(\hat{x},\hat{y})|_m^n=u_{0}u_{1}\cdots u_{m-1}(x,y)|_{m}^{n}u_{m-1}^{-1}\cdots u_{1}^{-1}u_{0}^{-1},$$
$$(\varphi(\hat{x}),\varphi(\hat{y}))|_m^n=(\hat{x},\hat{y})|_m^n\ker{(\phi)}.$$
Note that $d_G$ is two-side invariant. Then it follows from conditions (\romannumeral 4) and the above equality  that
$$\begin{aligned}
 d_\phi((\hat{x},\hat{y})|_m^n\ker(\phi),\ker(\phi))&=\inf\{d_G((\hat{x},\hat{y})|_m^n,h):h\in\ker(\phi)\}\\
 &=\inf\{d_G(({x},{y})|_{m}^{n},h):h\in\ker(\phi)\}\\
 &=d_\phi((x,y)|_{m}^{n}\ker(\phi),\ker(\phi)).
  \end{aligned}
$$
This shows that $\lim_m\sup_{m\leq n}d_\phi((\varphi(\hat{x}),\varphi(\hat{y}))|_m^n,\ker(\phi))=0$. Thus we have $\varphi(\hat{x})E_*(G_c/\ker{(\phi)})\varphi(\hat{x})$. It is clear that $\lim_n\hat{x}(n)\hat{y}(n)^{-1}=1_{G_c}$.

Next we show that $\hat{x}E_*(G_c)\hat{y}$.
Since $\ker(\phi)$ is non-archimedean, there exists a sequence of open subgroups $(M_l)$ of $\ker(\phi)$ which forms a neighborhood base of the identity element in $\ker(\phi)$. For any $\varepsilon>0$, there exist $l\in\omega$ and $0<\varepsilon'<\varepsilon$ such that
$$\{h\in\ker(\phi):d_G(1_G,h)<3\varepsilon'\}\subseteq M_l\subseteq\{h\in\ker(\phi):d_G(1_G,h)<\varepsilon\}.$$
Denote $\lim_n(\varphi(\hat{x}),\varphi(\hat{y}))|_0^n=g_0\ker(\phi)$. Then there is $N\in\omega$ such that
$$d_G(\hat{x}(n),\hat{y}(n))<\varepsilon',$$
$$d_\phi(\hat{x}(0)\cdots \hat{x}(n)\hat{y}(n)^{-1}\cdots \hat{y}(0)^{-1}\ker(\phi),g_0\ker(\phi))<\varepsilon'$$
for $n>N$. By the definition of $d_\phi$, there exist $e_n\in\ker(\phi)$ for each $n>N$ such that
$$d_G(\hat{x}(0)\cdots \hat{x}(n)\hat{y}(n)^{-1}\cdots \hat{y}(0)^{-1},g_0 e_n)<\varepsilon'.$$
Note that
$$\begin{array}{ll}& d_G(\hat{x}(0)\cdots \hat{x}(n)\hat{x}(n+1)\hat{y}(n+1)^{-1}\hat{y}(n)^{-1}\cdots \hat{y}(0)^{-1},\cr
&\quad\quad \hat{x}(0)\cdots \hat{x}(n)\hat{y}(n)^{-1}\cdots\hat{y}(0)^{-1})\cr
=& d_G(\hat{x}(n+1),\hat{y}(n+1))<\varepsilon'.\end{array}$$
Then $d_G(1_G,e_n^{-1}e_{n+1})=d_G(g_0 e_n,g_0 e_{n+1})<3\varepsilon'$. So $e_n^{-1}e_{n+1}\in M_l$. It follows that $e_n^{-1}e_m$ is in the open subgroup $M_l$ for $m\ge n>N$. Thus $d_G(e_n,e_m)<\varepsilon$, and hence
$$\begin{array}{ll}& d_G(\hat{x}(0)\cdots \hat{x}(n)\hat{y}(n)^{-1}\cdots \hat{y}(0)^{-1},\hat{x}(0)\cdots \hat{x}(m)\hat{y}(m)^{-1}\cdots\hat{y}(0)^{-1})\cr
<& 2\varepsilon'+d_G(g_0e_n,g_0e_m)<3\varepsilon.\end{array}$$
This gives $\hat{x}E_*(G_c)\hat{y}$.

Note that $(\hat{x},\hat{y})|_m^n=u_{0}u_{1}\cdots u_{m-1}(x,y)|_{m}^{n}u_{m-1}^{-1}\cdots u_{1}^{-1}u_{0}^{-1},$
 so $$\lim_m\sup_{m\leq n}d_G((x,y)|_{m}^{n},1_G)=\lim_m\sup_{m\leq n}d_G((\hat{x},\hat{y})|_m^n,1_G)=0.$$
Therefore ${x}E_*(G){y}$.
\end{proof}

\begin{theorem}\label{cor locall}
Let $G$ and $H$ be two TSI Polish groups such that $G$ is locally compact. Suppose $G_c$ is an open normal subgroup of $G$ and $H_c$ is a closed normal subgroup of $H$.
If there exists a continuous homomorphism $\phi:G_c\to H_c$ satisfying the following conditions:
\begin{enumerate}
  \item [(i)] $\ker(\phi)$ is non-archimedean.
  \item [(ii)] $\phi{\rm Inn}_G(G_c)\subseteq\overline{{\rm Inn}_H(H_c)\phi}$ under pointwise convergence topology.
\end{enumerate}
then $E(G)\le_BE(H)\times E(G;0)$. Moreover, if the interval $[0,1]$ embeds into $H$, then $E(G)\leq_B E(H)$.
\end{theorem}

\begin{proof}
Let $d_G,d_H$ be complete two-side invariant compatible metrics on $G$ and $H$ respectively.
It is clear that ${\rm Inn}_H(H_c)\phi$ is equicontinuous.
By Proposition~\ref{pointwis},
$\phi{\rm Inn}_G(G_c)$ is locally approximated by ${\rm Inn}_H(H_c)\phi$ on some open identity neighborhood $U$ of $G_c$. By shrinking we may assume $vUv^{-1}=U$ for all $v\in G$ and $U=U^{-1}$. For any $v\in G$, we have
$$\exists(h_n)\in H^\omega\,\lim_n\sup\{d_H(\phi(vg'v^{-1}),h_n\phi(g')h_n^{-1}):g'\in U\}=0.$$
 Let $G_c'=\bigcup_i U^i=\bigcup_i (vUv^{-1})^i\subseteq G_c$ and $\phi'=\phi|_{G'_c}$. Then $G_c'$ is an open normal subgroup of $G$. Note that $\phi'$ is a continuous homomorphism on $G'_c$, so $\phi'(vgv^{-1})=\lim_nh_n\phi'(g)h_n^{-1}$ for all $g\in G'_c$.
This implies that $$\phi'(g)=1_H\iff\lim_nd_H(h_n\phi'(g)h_n^{-1},1_H)=0\iff \phi(vgv^{-1})=1_H.$$
Thus $\iota_v(\ker(\phi'))=\ker(\phi')$ for any $v\in G$.
Let $\pi:G'_c\rightarrow G'_c/\ker{(\phi')}$ be the canonical projection map $\pi(g)=g\ker{(\phi')}$. Then $\pi$ is continuous and open. Define a map ${\widetilde{\phi}}:G'_c/\ker{(\phi')}\rightarrow H_c$ with ${\widetilde{\phi}}(g\ker{(\phi')})=\phi'(g)$. Then ${\widetilde{\phi}}$ is a continuous homomorphism. Since $G$ is locally compact, then there is an open subset $U\supseteq U_1\ni 1_G$ of $G'_c$ such that $\widetilde{\phi}|_{U_1\ker(\phi')}$ is a homeomorphism of $U_1\ker(\phi)$ to $\phi'(U_1)$. Note that $\phi'=\phi|_{G'_c}$, so $\phi'{\rm Inn}_G(G'_c)$ is locally approximated by ${\rm Inn}_H(H_c)\phi'$ on $U_1$.

Put $G_c=G'_c,H_c=H_c$ and $\phi=\phi'$. Then it follows from Lemma~\ref{the borl red} that $E(G)\le_BE(H)\times E(G;0)$.

Suppose the interval $[0,1]$ embeds into $H$. By~\cite[Lemma 2.9]{DZtsi}, we get that $E(G)\leq_B E(H)$.
\end{proof}
\section{Rigid theorem}
We say that a topological group $G$ \textit{has no small subgroups}, or \textit{is NSS}, if there exists an open subset $V\ni 1_G$ in $G$ such that no non-trivial subgroup of $G$ contained in $V$.
A Polish group $G$ is called \textit{strongly NSS} if there exists an open set $V\ni 1_G$ in $G$ such that
$$\forall(g_n)\in G^\omega\,(\lim_n g_n\ne 1_G\Rightarrow\exists n_0<\cdots<n_k\,(g_{n_0}\cdots g_{n_k}\notin V)),$$
where the set $V$ is called an \emph{unenclosed set} of $G$. All Lie groups are strongly NSS~(cf.~\cite[Proposition 5.2(2)]{DZtsi} and~\cite[Proposition 2.17]{HM07} ).

Now we begin to prove the main Rigid theorem.

Let $G,H$ be TSI Polish groups such that $H$ is a closed subgroup of $\prod_m H_m$, where all $H_m$ are separable TSI Lie groups. Then $H_m$ is also locally compact and strongly NSS.

Let $d_G,d_{H_m}\leq 1$ be compatible complete two-sided invariant metrics on $G$ and $H_m$ respectively.
A compatible complete two-sided invariant metric $d_H$ on $H$ is defined as
$$d_H(h,h')=\sum_{m=0}^\infty 2^{-m}d_{H_m}(h(m),h'(m))$$
for $h,h'\in H$.

For $m\in\omega$, let $\pi_m:\prod_m H_m\to H_m$ be the canonical projection map such that $\pi_m(h)=h(m)$ for $h\in\prod_m H_m$.
It is clear that every $\pi_m$ is a continuous homomorphism.

Assume that $E_*(G)\le_BE_*(H)$. Let $(F_n)$ be a sequence of finite subsets of $G$ such that
\begin{enumerate}
\item[(i)] $1_G\in F_n=F_n^{-1}$;
\item[(ii)] $F_{n-1}^{n}\subseteq F_n$; and
\item[(iii)] $\bigcup_nF_n$ is dense in $G$.
\end{enumerate}

Denote by $E$ the restriction of $E_*(G)$ on $\prod_nF_n$. By~\cite[Lemma 2.6]{DZtsi}, there exist an infinite set $I\subseteq\omega$ and a $w\in\prod_{n\notin I}F_n$ such that $E|_I^w\le_AE_*(H)$. So there are natural numbers $0=n_0<n_1<n_2<\cdots$ with $I=\{n_j:j\in\omega\}$, $0=l_0<l_1<l_2<\cdots$, maps $T_{n_j}:F_{n_j}\to H^{l_{j+1}-l_j}$, and $\psi:\prod_{n\in I}F_n\to H^\omega$ with
$$\psi(x)=T_{n_0}(x(n_0))^\smallfrown T_{n_1}(x(n_1))^\smallfrown T_{n_2}(x(n_2))^\smallfrown\cdots,$$
such that $\psi$ is an additive reduction of $E|_I^w$ to $E_*(H)$.

Put $u_{n_0}=1_G$, and for $j>0$, let
$$u_{n_j}=w(n_{j-1}+1)\cdots w(n_j-1).$$
It follows from (i) and (ii) that $u_{n_j}^{-1}\in F_{n_j}$ for each $j\in\omega$.

Define $x_0\in\prod_{n\in I}F_n$ as
$$x_0(n_j)=u_{n_j}^{-1}\quad(\forall j\in\omega).$$
We may assume that $\psi(x_ 0)=1_{H^\omega}$  (see Page 9 of~\cite{DZtsi}). Then $T_{n_j}(u_{n_j}^{-1})=1_H$.

For $s=(h_0,\ldots,h_{l-1})$ and $t=(h_0',\ldots,h_{l-1}')$ in $H^l$, we define
$$d_H^\infty(s,t)=\max_{0\le i\le m<l}d_H(h_i\cdots h_m,h_i'\cdots h_m').$$

For any integer $j>0$ and $g\in F_{n_j-1}$, we have
$$u_{n_j}^{-1}g=w(n_j-1)^{-1}\cdots w(n_{j-1}+1)^{-1}g\in F_{n_j-1}^{n_j}\subseteq F_{n_j}.$$

\begin{lemma}[{\cite[Lemma 3.1]{DZtsi}}]\label{continuous property}
For any $q\in\omega$, there exists a $\delta_q>0$ such that
$$\forall^\infty n\in I\,\forall g,g'\in F_{n-1}\,(d_G(g,g')<\delta_q\Rightarrow d_H^\infty(T_n(u_n^{-1}g),T_n(u_n^{-1}g'))<2^{-q}).$$
\end{lemma}

\begin{definition}[{\cite[Definition 3.2]{DZtsi}}]\label{S_n_j}
For any $g\in\bigcup_nF_n$, if $g\in F_n$ for some $n<n_j$, then $u_{n_j}^{-1}g\in F_{n_j}$, so $T_{n_j}(u_{n_j}^{-1}g)\in H^{l_{j+1}-l_j}$. This allows us to define
$$S_{n_j}(g)=T_{n_j}(u_{n_j}^{-1}g)(0)\cdots T_{n_j}(u_{n_j}^{-1}g)(l_{j+1}-l_j-1)\in H.$$
\end{definition}

Recall that $\psi(x_ 0)=1_{H^\omega}$, this implies that
$$S_{n_j}(1_G)=T_{n_j}(u_{n_j}^{-1})(0)\cdots T_{n_j}(u_{n_j}^{-1})(l_{j+1}-l_j-1)=1_H$$
for all $j\in\omega$. It is clear that, for $g,g'\in F_n$ with $n<n_j$,
$$d_H(S_{n_j}(g),S_{n_j}(g'))\le d_H^\infty(T_{n_j}(u_{n_j}^{-1}g),T_{n_j}(u_{n_j}^{-1}g')).$$

By~\cite[Theorem 6.4]{DZtsi}, there exists a continuous homomorphism $S:G_0\rightarrow H$ such that $\ker(S)$ is non-archimedean. The map $S$ is defined by
$$S(g)=(\pi^S_0(g),\pi^S_1(g),\ldots,\pi^S_m(g),\ldots),$$
where $\pi^S_m$ is an uniformly continuous homomorphism from $G_0$ to $H_m$ satisfying $\pi^S_m(g)=\lim_{j}\pi_m(S_{n_{j}}(g))$ for all $g\in G_0\cap\bigcup_n F_n$.

For each $m\in\omega$, by~\cite[Exercise 2.1.4]{gaobook},
we can find a compact neighborhood $K_m\subseteq(H_m)_0$ of $1_{H_m}$ and $\varepsilon_m^*>0$ such that $\{h\in H_m: d_{H_m}(h,1_{H_m})\leq\varepsilon_m^*\}\subseteq K_m$ and $hK_mh^{-1}=K_m$ holds for each $h\in H_m$, where $(H_m)_0$ is the connected component of $1_{H_m}$ in $H_m$.

Next we prove some technical lemmas.
\begin{lemma}\label{lem K_m}
Let $m\in\omega$, and let $(f_n)$ be a sequence of isometries of $(K_m, d_{H_m})$ onto itself. Then there are natural numbers $k_0<k'_0<k_1<k'_1\cdots<k_{i}<k'_{i}<\cdots$ such that
$$\sup_{h\in K_m}d_{H_m}(f_{k'_i}\circ f_{k'_i-1}\circ\cdots f_{k_i+1}\circ f_{k_i}(h),h)<2^{-i}.$$
\end{lemma}

\begin{proof}
We only need to show that
$$\forall\varepsilon>0\,\forall M\,\exists q'>q>M\,(\sup_{h\in K_m}d_{H_m}(f_{q'}\circ f_{q'-1}\circ\cdots f_{q+1}\circ f_{q}(h),h)<\varepsilon).$$

Fix a real number $\varepsilon>0$ and a natural number $M$.  By the completeness of $K_m$, we can find  finitely many elements $h_0,h_1,\cdots h_p$ in $K_m$ such that, for any $h\in K_m$, there is some $h_i$ with $d_{H_m}(h,h_i)<4^{-1}{\varepsilon}$.

For $i\leq p$ and $n\in\omega$, denote $s(h_i,n)=f_{n}\circ f_{n-1}\circ\cdots f_{1}\circ f_{0}(h_i)\in K_m$. Since $K_m$ is compact, there are $j_0<j_1<\cdots<j_n<\cdots$ so that $\lim_ns(h_i,j_n)$ exists for all $i\leq p$. Then by the completeness of $d_H$, there are $j_{k},j_{k'}$ with $M<j_{k}<j_{k'}-1$ and
$$\forall i\leq p\,(d_{H_m}(s(h_i,j_k),s(h_i,j_{k'}))<4^{-1}\varepsilon).$$

Let $h\in K_m$, and let $h_i\in K_m$ with $d_{H_m}(h,h_i)<4^{-1}\varepsilon$.
Note that all $f_n$ are isometric, so
$$\begin{aligned}
d_{H_m}(s(h,j_k),s(h,j_{k'}))&\leq d_{H_m}(s(h,j_k),s(h_i,j_k))+ d_{H_m}(s(h_i,j_k),s(h_i,j_{k'}))\\
&+d_{H_m}(s(h,j_{k'}),s(h_i,j_{k'}))\leq 3\times4^{-1}\varepsilon<\varepsilon.\\
  \end{aligned}$$
And thus $d_{H_m}(f_{j_{k}}\circ f_{j_{k}-1}\circ\cdots f_{1}\circ f_{0}(h),f_{j_{k'}}\circ f_{j_{k'}-1}\circ\cdots f_{1}\circ f_{0}(h))<\varepsilon$ for all $h\in K_m$. Again since all $f_n$ are surjective, we have $$d_{H_m}(f_{j_{k'}}\circ f_{j_{k'}-1}\circ\cdots f_{j_k+2}\circ f_{j_k+1}(h),h)<\varepsilon\,(\forall h\in K_m).$$
Thus $q=j_k+1$ and $q'=j_{k'}$ are as required.
\end{proof}

Now fix some $m\in\omega$. Since $H_m$ is strongly NSS, there is a positive real number $\varepsilon^*<\varepsilon^*_m$ so that $V=\{h\in H_m:d_{H_m}(h,1_{H_m})<\varepsilon^*\}$ is an unenclosed set of $H_m$.

Then by Lemma~\ref{continuous property}, there is a $\delta_m^*>0$ such that
$$\forall^\infty n\in I\,\forall g\in F_{n-1}\,(d_G(g,1_G)<\delta_m^*\Rightarrow d_{H_m}(\pi_m(S_{n}(g)),1_{H_m})<4^{-1}\varepsilon^*).$$

Denote $X_m=\{g\in G_0\cap\bigcup_n F_n:d_G(g,1_G)<\delta_m^*\}$.

\begin{lemma}\label{lemm g conge}
Let $g\in X_m$ and $M\in\omega$. For any $v\in \bigcup_n F_n$,
if $$\lim_j \pi_m(S_{n_{j-1}}(vgv^{-1})S_{n_{j}}(v)S_{n_{j+1}}(g^{-1})S_{n_{j}}(v)^{-1})\neq1_{H_m},$$ then there are $l(0),l(1),\cdots,l(q)\in\omega$ and a sequence of natural numbers $p(i,0)<p(i,1)<\cdots <p(i,l(i))$ for $i\leq q$ such that
\begin{enumerate}
\item [(\romannumeral 1)]$M+2<p(0,0),\,p(i,0)+1<p(i,1)$ and $p(i,l(i))+1<p(i+1,0);$
  \item[(\romannumeral 2)] $1_G,v,g\in F_{n_{p(0,0)-3}};$
  \item [(\romannumeral 3)]for brevity, we denote $$A^{i}=\pi_m(S_{n_{p(i,0)-1}}(vgv^{-1})S_{n_{p(i,0)}}(v)S_{n_{p(i,0)+1}}(g^{-1})),$$
      $$B^{i}=\pi_m(S_{n_{p(i,1)}}(v)S_{n_{p(i,2)}}(v)\cdots S_{n_{p(i,l(i))}}(v)),$$
then $$d_{H_m}(A^{0}B^{0}\cdots A^{q}B^{q},\pi_m(S_{n_{p(0,0)}}(v))B^{0}\cdots \pi_m(S_{n_{p(q,0)}}(v))B^{q})>4^{-1}\varepsilon^*.$$

\end{enumerate}
\end{lemma}

\begin{proof}
First we can find a real number $\epsilon>0$ and a sequence of natural numbers $M+2<j(0)<j(1)<\cdots<j(k)<\cdots$ such that $1_G,v,g\in F_{n_{j(0)-3}}$, $j(k)+1<j(k+1)$ and $$d_{H_m}(\pi_m(S_{n_{j(k)-1}}(vgv^{-1})S_{n_{j(k)}}(v)S_{n_{j(k)+1}}(g^{-1}))S_{n_{j(k)}}(v)^{-1}),1_{H_m})>\epsilon.$$
Note that $\pi^S_m$ is continuous, $\pi^S_m(g')=\lim_{j}\pi_m(S_{n_{j}}(g'))$ for all $g'\in G_0$ and $g, vgv^{-1}\in G_0,$
so $\pi^S_m(g),\pi^S_m(vgv^{-1})\in(H_m)_0$~(the connected component of $1_{H_m}$ in $H_m$). Thus we may assume that $$\pi_m(S_{n_{j(k)+1}}(g^{-1})),\pi_m(S_{n_{j(k)-1}}(vgv^{-1}))\in(H_m)_0~(\forall k\in\omega).$$ This is because, $(H_m)_0$ is open in $H_m$.

For $k\in\omega$, let $f_k:(K_m, d_{H_m})\rightarrow (K_m, d_{H_m})$ be the isometry with $f_k(h)=\pi_m(S_{n_{j(k)}}(v))^{-1}h\pi_m(S_{n_{j(k)}}(v))$. Put an $i^*\in\omega$ with $2^{-i^*}<4^{-1}\varepsilon^*$.
Then it follows from Lemma~\ref{lem K_m} that
there are natural numbers $k_0<k'_0<k_1<k'_1\cdots<k_{i}<k'_{i}<\cdots$ such that
$$\sup_{h\in K_m}d_{H_m}(f_{k'_i}\circ f_{k'_i-1}\circ\cdots f_{k_i+1}\circ f_{k_i}(h),h)<2^{-(i+i^*)}.\eqno{(*)}$$

For $i\in\omega$, denote   $$A_{i}=\pi_m(S_{n_{j(k_{2i})-1}}(vgv^{-1})S_{n_{j(k_{2i})}}(v)S_{n_{j(k_{2i})+1}}(g^{-1})),$$
      $$B_{i}=\pi_m(S_{n_{j(k_{2i}+1)}}(v)S_{n_{j(k_{2i}+2)}}(v)\cdots S_{n_{j(k'_{2i})}}(v)).$$

Assume toward a contradiction that, for any natural numbers $t(q)>t(q-1)>\cdots>t(0)>i^*$, denote $h^*_{t(i)}=\pi_m(S_{n_{j(k_{2t(i)})}}(v))B_{t(i)}$, then we have
$$d_{H_m}(A_{t(0)}B_{t(0)}\cdots A_{t(q)}B_{t(q)},h^*_{t(0)}\cdots h^*_{t(q)})\leq 4^{-1}{\varepsilon^*}.\eqno{(**)}$$

For each $i\in\omega$, denote $h_i=A_{i}\pi_m(S_{n_{j(k_{2i})}}(v)^{-1})$. It is clear that $\lim_i h_i\neq 1_{H_m}$. Recall the following fact,
 $$\forall^\infty n\in I\,\forall g\in F_{n-1}\,(d_G(g,1_G)<\delta_m^*\Rightarrow d_{H_m}(\pi_m(S_{n}(g)),1_{H_m})<4^{-1}\varepsilon^*).$$
Since $g, vgv^{-1}\in X_m=\{g\in G_0\cap\bigcup_n F_n:d_G(g,1_G)<\delta_m^*\}$, there is some $i_1^*\in\omega$ so that $d_{H_m}(h_i,1_{H_m})<2^{-1}\varepsilon^*$ for all $i\geq i_1^*$.

Note that $\pi_m(S_{n_{j(k)+1}}(g^{-1})),\pi_m(S_{n_{j(k)-1}}(vgv^{-1}))\in(H_m)_0$ for all $k\in\omega$, so $h_i\in (H_m)_0$.
Now we fix some $l(i)>l(i-1)>\cdots>l(0)>i^*+i^*_1$, then it follows from condition $(*)$ that
$$\sup_{h\in K_m}d_{H_m}(B_{l(i')}^{-1}\pi_m(S_{n_{j(k_{2l(i')})}}(v)^{-1})h\pi_m(S_{n_{j(k_{2l(i')})}}(v))B_{l(i')},h)<2^{-(i'+i^*)},$$
where $i'\leq i$.
Denote $s(i')=A_{l(0)}B_{l(0)}\cdots A_{l(i')}B_{l(i')}$ and
$$\hat{s}(i')=\pi_m(S_{n_{j(k_{2l(0)})}}(v))B_{l(0)}\cdots \pi_m(S_{n_{j(k_{2l(i')})}}(v)) B_{l(i')}.$$
By condition $(**)$ and the fact that $\{h\in H_m: d_{H_m}(h,1_{H_m})\leq\varepsilon_m^*\}\subseteq K_m$, we have
$$d_{H_m}(\hat{s}(i')^{-1}s(i'),1_{H_m})\leq4^{-1}{\varepsilon^*},\quad \hat{s}(i')^{-1}s(i')\in K_m\subseteq(H_m)_0,$$
$$d_{H_m}(\hat{s}(i'-1)^{-1}s(i'-1)h_{l(i')},1_{H_m})<{\varepsilon^*},\,\hat{s}(i'-1)^{-1}s(i'-1)h_{l(i')}\in K_m,$$
Note that
$$\hat{s}(i')^{-1}s(i')=B_{l(i')}^{-1}\pi_m(S_{n_{j(k_{2l(i')})}}(v)^{-1})\hat{s}(i'-1))^{-1}s(i'-1)A_{l(i')}B_{l(i')},$$ $$\hat{s}(i'-1))^{-1}s(i'-1)h_{l(i')}=\hat{s}(i'-1))^{-1}s(i'-1)A_{l(i')}\pi_m(S_{n_{j(k_{2l(i')})}}(v)^{-1})\in K_m$$
and $$\sup_{h\in K_m}d_{H_m}(B_{l(i')}^{-1}\pi_m(S_{n_{j(k_{2l(i')})}}(v)^{-1})h\pi_m(S_{n_{j(k_{2l(i')})}}(v))B_{l(i')},h)<2^{-(i'+i^*)} ,$$
so
$$
d_{H_m}(\hat{s}(i')^{-1}s(i'),\hat{s}(i'-1))^{-1}s(i'-1)h_{l(i')})
<2^{-(i'+i^*)}.
$$
This implies that $$d_{H_m}(\hat{s}(i)^{-1}s(i),\hat{s}(0)^{-1}s(0)h_{l(1)}h_{l(2)}\cdots h_{l(i)})<\sum_{i'\leq i}2^{-(i'+i^*)}<2^{-1}\varepsilon_*.$$
By the fact that $d_{H_m}(\hat{s}(i')^{-1}s(i'),1_{H_m})\leq 4^{-1}{\varepsilon^*}$, we have
$$
\begin{aligned}
d_{H_m}(h_{l(1)}h_{l(2)}\cdots h_{l(i)},1_{H_m})&\leq d_{H_m}(\hat{s}(i)^{-1}s(i),\hat{s}(0)^{-1}s(0))+2^{-1}\varepsilon_*\cr
&\leq4^{-1}{\varepsilon^*}+4^{-1}{\varepsilon^*}+2^{-1}\varepsilon_*\cr
&\leq\varepsilon_*.
\end{aligned}
$$

In other words, for any $l(i)>l(i-1)>\cdots>l(1)>i^*$, we get $h_{l(1)}h_{l(2)}\cdots h_{l(i)}\in V$, but $\lim_i h_i\neq 1_{H_m}$. This contradicts the assumption that $V$ is an unenclosed set of $H_m$.

Now we can find $t(q)>t(q-1)>\cdots>t(0)>i^*$ such that
$$d_{H_m}(A_{t(0)}B_{t(0)}\cdots A_{t(q)}B_{t(q)},h^*_{t(0)}\cdots h^*_{t(q)})> 4^{-1}{\varepsilon^*},$$
where $h^*_{t(i)}=\pi_m(S_{n_{j(k_{2t(i)})}}(v))B_{t(i)}$. Put $p(i,0)=j({k_{2t(i)}})$, $l(i)=k'_{2t(i)}-k_{2t(i)}$ and $p(i,l)=j({k_{2t(i)}+l})$ for $l\leq l(i),i\leq q$. Then $A^i=A_{t(i)}$ and $B^i=B_{t(i)}$.
Thus (i) through (iii) are satisfied.
\end{proof}

\begin{lemma}\label{lemm pre conjuat}
There is a real number $0<\delta<\delta_m^*$ so that, for any $g\in X_m$ and $v\in \bigcup_n F_n$,  we have
$$d_G(g,1_G)<\delta\Rightarrow \lim_j \pi_m(S_{n_{j-1}}(vgv^{-1})S_{n_{j}}(v)S_{n_{j+1}}(g^{-1})S_{n_{j}}(v)^{-1})=1_{H_m}.$$
\end{lemma}

\begin{proof}
If not, then there are a sequence $(g_k)$ in $X_m$ and $(v_k)\in (\bigcup_n F_n)^\omega$ satisfying
\begin{enumerate}
  \item [(a)] $d_{G}(g_k,1_G)<2^{-k}$;
  \item [(b)] $\lim_j \pi_m(S_{n_{j-1}}(v_kg_kv_k^{-1})S_{n_{j}}(v_k)S_{n_{j+1}}(g_k^{-1})S_{n_{j}}(v_k)^{-1}\neq1_{H_m}.$
\end{enumerate}
Then by Lemma~\ref{lemm g conge}, we can find $l(k,0),l(k,1),\cdots,l(k,q_k)\in\omega$ and a sequence of natural numbers $p(k,i,0)<p(k,i,1)<\cdots <p(k,i,l(k,i))$ for $i\leq q_k$ such that
\begin{enumerate}
  \item [(\romannumeral 1)] $p(k,q_k,l(k,q_k))+3<p(k+1,0,0)$, $p(k,i,0)+1<p(k,i,1)$ and $p(k,i,l(k,i))+1<p(k,i+1,0)$;
  \item [(\romannumeral 2)] $1_G,v_k,g_k\in F_{n_{p(k,0,0)-3}}$;
  \item [(\romannumeral 3)]for brevity, we denote $$A^{i}_k=\pi_m(S_{n_{p(k,i,0)-1}}(v_kg_kv_k^{-1})S_{n_{p(k,i,0)}}(v_k)S_{n_{p(k,i,0)+1}}(g_k^{-1})),$$
      $$B^{i}_k=\pi_m(S_{n_{p(k,i,1)}}(v_k)S_{n_{p(k,i,2)}}(v_k)\cdots S_{n_{p(k,i,l(k,i))}}(v_k)),$$
then $$d_{H_m}(A^{0}_kB^{0}_k\cdots A^{q_k}_kB^{q_k}_k,\pi_m(S_{n_{p(k,0,0)}}(v_k))B^{0}_k\cdots \pi_m(S_{n_{p(k,q_k,0)}}(v_k))B^{q_k}_k)>4^{-1}\varepsilon^*.$$
\end{enumerate}

For each $n\in I$, define
$$\begin{aligned}
  x(n)&=\left\{\begin{array}{ll}
   u_n^{-1}v_kg_kv_k^{-1}, & n=n_{p(k,i,0)-1},i\leq q_k,\cr
  u_n^{-1} v_k, &  n=n_{p(k,i,j)},0\leq j\leq l(k,i),i\leq q_k, \cr
     u_n^{-1}g_k^{-1}, & n=n_{p(k,i,0)+1},i\leq q_k,\cr
   u_n^{-1},& \mbox{otherwise},
  \end{array}\right.\\
y(n)&=\left\{\begin{array}{ll}
      u_n^{-1} v_k, &  n=n_{p(k,i,j)},0\leq j\leq l(k,i),i\leq q_k, \cr
   u_n^{-1},& \mbox{otherwise}.
  \end{array}\right.
\end{aligned}
$$
For any $q>n_{p(0,0,0)}$, the element $$(x\oplus w,y\oplus w)|_0^q\mbox{ is equal to }1_G \mbox{ or } v_0^{a_0}\cdots v_k^{a_k}g_kv_k^{-a_k}\cdots v_0^{-a_0},$$where $k,a_0,a_1,\cdots,a_k$ are some natural numbers.

 Note that $$d_G(v_0^{a_0}\cdots v_k^{a_k}g_kv_k^{-a_k}\cdots v_0^{-a_0},1_G)=d_{G}(g_k,1_G)<2^{-k},$$ so $(x\oplus w) E_*(G)(y\oplus w)$.

Then we have
$$\pi_m(\psi(x)(l_{p(k,i,0)-1})\psi(x)(l_{p(k,i,0)-1}+1)\cdots\psi(x)(l_{p(k,i,0)+2}-1))=A_k^i.$$
$$\psi(y)(l_{p(k,i,0)-1})\psi(y)(l_{p(k,i,0)-1}+1)\cdots\psi(y)(l_{p(k,i,0)+2}-1)=S_{n_{p(k,i,0)}}(v_k).$$
$$\pi_m(\psi(x)(l_{p(k,i,0)+2})\psi(x)(l_{p(k,i,0)+2}+1)\cdots\psi(x)(l_{p(k,i,l(k,i))+1}-1))=B_k^i.$$
$$\pi_m(\psi(y)(l_{p(k,i,0)+2})\psi(y)(l_{p(k,i,0)+2}+1)\cdots\psi(y)(l_{p(k,i,l(k,i))+1}-1))=B_k^i.$$
Note that $T_{n_j}(u_{n_j}^{-1})=1_H$ for all $j\in\omega$, so
$$d_{H_m}(\pi_m((\psi(x),\psi(y))|_{l_{p(k,0,0)-1}}^{l_{p(k,q_k,l(k,q_k))+1}-1}),1_{H_m})>4^{-1}\varepsilon^*.$$
This implies that $\psi(x) E_*(H)\psi(y)$ fails. This contradicts the fact that $\psi$ is a reduction.
\end{proof}

\begin{lemma}\label{lemm mconve}
For any $g\in G_0$ and $v\in \bigcup_n F_n$, we have
$$\pi^S_m(vgv^{-1})=\lim_j\pi_m(S_{n_j}(v))\pi^S_m(g)\pi_m(S_{n_j}(v))^{-1}).$$
\end{lemma}

\begin{proof}
By Lemma~\ref{lemm pre conjuat}, there is a real number $0<\delta<\delta_m^*$ so that, for any $g\in X_m=\{g\in G_0\cap\bigcup_n F_n:d_G(g,1_G)<\delta_m^*\}$,
$$d_G(g,1_G)<\delta\Rightarrow \lim_j \pi_m(S_{n_{j-1}}(vgv^{-1})S_{n_{j}}(v)S_{n_{j+1}}(g^{-1})S_{n_{j}}(v)^{-1})=1_{H_m}.$$
Let $U=\{g\in G_0:d_G(g,1_G)<\delta\}$.
Recall that $$\pi^S_m(g)=\lim_{j}\pi_m(S_{n_{j}}(g))$$ for all $g\in G_0\cap\bigcup_n F_n$, so
$$\pi^S_m(vgv^{-1})=\lim_j\pi_m(S_{n_j}(v))\pi^S_m(g)\pi_m(S_{n_j}(v))^{-1})\,(\forall g\in U\cap X_m).$$
Note that $\bigcup_n F_n$ is dense in $G$. Then by the continuity of $\pi^S_m$, we have $$\pi^S_m(vgv^{-1})=\lim_j\pi_m(S_{n_j}(v))\pi^S_m(g)\pi_m(S_{n_j}(v))^{-1})\,(\forall g\in U).$$
Then $\pi^S_m(vgv^{-1})=\lim_j\pi_m(S_{n_j}(v))\pi^S_m(g)\pi_m(S_{n_j}(v))^{-1})$ for all $g\in G_0$. This is because, $\pi^S_m$ is a homomorphism and $G_0$ is connected~($G_0=\bigcup_i U^i$).
\end{proof}

\begin{lemma}
For any $g\in G_0$, we have
$$S(vgv^{-1})=\lim_jS_{n_j}(v)S(g)S_{n_j}(v)^{-1}.$$
\end{lemma}

\begin{proof}
Note that the map $S: G_0\rightarrow H$ is defined by
$$S(g)=(\pi^S_0(g),\pi^S_1(g),\ldots,\pi^S_m(g),\ldots).$$
Then this follows immediately from Lemma~\ref{lemm mconve}, since $H$ is equipped with the pointwise convergence topology.
\end{proof}

\begin{theorem}\label{the 1}
Let $G$ be a TSI Polish group such that $G_0$ is open in $G$, and let $H$ be a pro-Lie TSI Polish group. If $E(G)\leq_BE(H)$, then there exists a continuous homomorphism $\phi:G_0\rightarrow H$ such that $\ker(\phi)$ is non-archimedean and
 $\phi{\rm Inn}_G(G_0)\subseteq\overline{{\rm Inn}_H(H_0)\phi}$ under pointwise convergence topology.
\end{theorem}

\begin{proof}
We fix a set $T_G\subseteq G$ such that $T_G$ meets each coset of $G_0$ at exactly one point.
Suppose $E(G)\leq_BE(H)$. By the preceding arguments and Lemmas, there exists a continuous homomorphism $S:G_0\rightarrow H$ such that $\ker(S)$ is non-archimedean and
$$S(vgv^{-1})=\lim_jS_{n_j}(v)S(g)S_{n_j}(v)^{-1}$$
for all $v\in T_G$ and $g\in G_0$. Let $u\in G$ with $u=vg_*$, where $v\in T_G,g_*\in G_0$.
Then for any $g\in G_0$,
$$\begin{aligned}
S(ugu^{-1})=S(vg_*gg_*^{-1}v^{-1})&=\lim_jS_{n_j}(v)S(g_*gg_*^{-1})S_{n_j}(v)^{-1} \\
&=\lim_jS_{n_j}(v)S(g_*)S(g)S(g_*)^{-1}S_{n_j}(v)^{-1}.
  \end{aligned}$$
This shows that  $S{\rm Inn}_G(G_0)\subseteq\overline{{\rm Inn}_H(H_0)S}$.
So $\phi=S$ is as required.
\end{proof}

\begin{theorem}[Rigid theorem on locally compact TSI Polish groups admitting open identity component]\label{main the}
Let $G$ be a locally compact TSI Polish group such that $G_0$ is open in $G$, and let $H$ be a nontrivial pro-Lie TSI Polish group. Then $E(G)\leq_BE(H)$ iff there exists a continuous homomorphism $\phi:G_0\to H_0$ satisfying the following conditions:
\begin{enumerate}
  \item [(i)] $\ker(\phi)$ is non-archimedean;
  \item [(ii)]  $\phi{\rm Inn}_G(G_0)\subseteq\overline{{\rm Inn}_H(H_0)\phi}$ under pointwise convergence topology.
\end{enumerate}
\end{theorem}

\begin{proof}
If $G_0=\{1_G\}$, then $G$ is a countable discrete group. Note that $H$ is nontrivial. Then by Theorem 3.5 of~\cite{DZ}, we have $E_0\sim_B E(G)\leq_B E(H)$. In this case, we let $\phi:G_0\rightarrow H$ with $\phi(1_G)=1_H$.

Suppose $G_0\neq\{1_G\}$. $(\Rightarrow)$. It follows by Theorem~\ref{the 1}.

$(\Leftarrow)$. Let $\phi:G_0\to H_0$ be a continuous homomorphism satisfying conditions (i) and (ii). The condition (i) implies that $H_0$ is nontrivial. Thus the interval [0,1] embeds into $H_0$~(cf.~\cite[Proposition 19]{HM07}).
Then by Theorem~\ref{cor locall}, we get $E(G)\leq_B E(H)$.
\end{proof}

A similar proof gives the following theorem.

\begin{theorem}\label{loca open}
Let $G$ be a TSI Polish group, and let $H$ be a separable TSI Lie group with $\dim(H)\geq 1$.
 If $E(G)\leq_BE(H)$, then there are an open normal subgroup $G_c$ of $G$ and a continuous homomorphism $\phi:G_c\to H$ satisfying the following conditions:
\begin{enumerate}
  \item [(i)] $\ker(\phi)$ is non-archimedean;
  \item [(ii)]  $\phi{\rm Inn}_G(G_c)\subseteq\overline{{\rm Inn}_H(H)\phi}$ under pointwise convergence topology.
  \item [(iii)] the map $\phi^\#$ is a Borel reduction of $E_*(G_c)$ to $E_*(H)\times E(G_c;0)$, where $\phi^\#:G_c^\omega\rightarrow H^\omega\times G_c^\omega$ satisfies $\phi^\#(x)(n)=(\phi(x(n)),x(n))$.
\end{enumerate}
Moreover, if $G$ is locally compact,  then the converse is also true.
\end{theorem}

\begin{proof}
Let $d_G,d_{H}\leq 1$ be compatible complete two-sided invariant metrics on $G$ and $H$ respectively. Assume that $E_*(G)\le_BE_*(H)$.
We use the notation defined in the arguments before Definition~\ref{S_n_j}. By~\cite[Remark 6.5]{DZtsi}, there exist an open normal subgroup $W \supseteq G_0$ of $G$ and an uniformly  continuous homomorphism $S':W\rightarrow H$ such that $\ker(S')$ is non-archimedean, where $S'(g)=\lim_{j}S_{n_{j}}(g)$ for all $g\in W\cap\bigcup_n F_n$. The map ${(S')}^\#$ is a reduction of $E_*(W)$ to $E_*(H)\times E(W;0)$.

Then fix a compact neighborhood $K'\subseteq H$ of $1_H$ and some $\epsilon_*>0$ so that $V=\{h\in H: d_H(h,1_H)\leq\epsilon_*\}\subseteq K'$ and $hK'h^{-1}=K'$ holds for each $h\in H$.  Since $H$ is strongly NSS, we let $\epsilon_*$ be sufficiently small so that $V$ is an unenclosed set of $H$.

By Lemma~\ref{continuous property}, we can find a $\delta_*>0$ such that
$$\forall^\infty n\in I\,\forall g\in F_{n-1}\,(d_G(g,1_G)<\delta_*\Rightarrow d_{H}(S_{n}(g),1_{H})<8^{-1}\epsilon_*).$$

Denote $X=\{g\in W\cap\bigcup_n F_n:d_G(g,1_G)<\delta_*\}$. We repeat the proofs of Lemmas~\ref{lem K_m},~\ref{lemm g conge} and Lemma~\ref{lemm pre conjuat} to establish the following fact:

There is a real number $0<\delta'<\delta_*$ so that, for any $g\in X$ and $v\in \bigcup_n F_n$,
$$d_G(g,1_G)<\delta'\Rightarrow \lim_j S_{n_{j-1}}(vgv^{-1})S_{n_{j}}(v)S_{n_{j+1}}(g^{-1})S_{n_{j}}(v)^{-1}=1_{H}.$$
We set $U'=\{g\in W:d_G(g,1_G)<\delta'\}$ and $G_c=\bigcup_i (U')^i$. By the continuity of $S'$, we have $S'(vgv^{-1})=\lim_j S_{n_j}(v)S'(g)S_{n_j}(v)^{-1}$ for all $g\in G_c,v\in\bigcup_n F_n$. Let $\phi=S'|_{G_c}$. Note that $G_c$ is an open normal subgroup of $G$.
Since $\bigcup_n F_n$ is dense, we can find a set $T'_G\subseteq \bigcup_n F_n$ such that $T'_G$ meets each coset of $G_c$ at exactly one point. A similar proof of Theorem~\ref{the 1} gives that $\phi{\rm Inn}_G(G_c)\subseteq\overline{{\rm Inn}_H(H)\phi}$ under pointwise convergence topology. So $\phi$ and $G_c$ are as required.

On the other hand, we suppose $G$ is locally compact. Let $\phi:G_c\to H$ be a continuous homomorphism satisfying conditions (i) and (ii). By Theorem~\ref{cor locall}, we have $E(G)\leq_B E(H)$.
\end{proof}

All Lie groups are pro-Lie and any non-archimedean closed subgroup of a Lie group is countable.
Using Theorem~\ref{loca open} and these facts, we obtain the following result.

\begin{corollary}\label{Rn}
Let $G$ and $H$ be two separable TSI Lie groups. Then $E(G)\leq_BE(H)$ iff there exists a continuous locally injective homomorphism $\phi:G_0\rightarrow H_0$ such that $\phi{\rm Inn}_G(G_0)\subseteq\overline{{\rm Inn}_H(H_0)\phi}$ under pointwise convergence topology. If $\dim(G)=\dim(H)$, then the map $\phi$ is surjective. Moreover,
$\phi$ is also a topological group isomorphism when $G_0\cong H_0\cong\R^n$.
\end{corollary}

\begin{proof}
Suppose $\dim(G)=\dim(H)$. By~\cite[Lemma 1]{MO}, any continuous locally injective homomorphism $\phi:G_0\rightarrow H_0$ is open. Note that $H_0$ is connected, so $\phi$ is also surjective. Next we assume $G_0\cong H_0\cong \R^n$.

Fix an open neighborhood $V\subseteq G_0$ of $1_G$ such that $\phi|_{V}$ is injective. For any $g\in G_0$, we can find an element $g'\in V$ and a natural number $n$ with $g=(g')^n$.  Since $G_0\cong\R^n$, we have $\phi(g)=1_H\iff \phi(g')=1_H\iff g'=1_G\iff g=1_G$. This shows that $\phi$ is injective, and thus $\phi$ is a topological isomorphism.
Then it follows by the above facts and Theorem~\ref{loca open}.
\end{proof}
\begin{lemma}\label{the fint}
Let $G$ and $H$ be two TSI Polish groups, and let $\phi:G_0\rightarrow H_0$ be a continuous locally injective homomorphism. Suppose $H$ is compact or $|{\rm Inn}_H(H_0)|<+\infty$. Then $\phi{\rm Inn}_G(G_0)\subseteq\overline{{\rm Inn}_H(H_0)\phi}$ under pointwise convergence topology  iff $\phi{\rm Inn}_G(G_0)\subseteq {\rm Inn}_H(H_0)\phi$.
\end{lemma}

\begin{proof}
Its proof is trivial.
\end{proof}

For the case of $G_0\cong \R^n$, we have a positive answer to Question 7.5 of~\cite{DZ}.

\begin{corollary}
\label{the Rn finit}
Let $G$ and $H$ be two separable TSI Lie groups such that $G_0\cong H_0\cong \R^n$ and $|{\rm Inn}_H(H_0)|<+\infty$. Then $E(G)\leq_BE(H)$ iff there exists a topological isomorphism $\phi:G_0\rightarrow H_0$ such that $\phi{\rm Inn}_G(G_0)\phi^{-1}\subseteq {\rm Inn}_H(H_0)$.
\end{corollary}
\begin{proof}
It follows by Theorem~\ref{Rn} and Lemma~\ref{the fint}.
\end{proof}

The following result provide a negative answer to Question 7.5 of~\cite{DZ}.
\begin{theorem}
Let $G=\Gamma_k\ltimes\R^2$ and $H={\Gamma}^z\ltimes\R^2$. Then $E(G)<_B E(H)$, but there is no topological isomorphism $\phi:G_0\to H_0$ such that $\phi{\rm Inn}_G(G_0)\phi^{-1}\subseteq{\rm Inn}_H(H_0)$.
\end{theorem}

\begin{proof}
Clearly, we have ${\rm Inn}_G(G_0)\subseteq\overline{{\rm Inn}_H(H_0)}$ and $k=|{\rm Inn}_G(G_0)|<|{\rm Inn}_H(H_0)|=\infty$. By Corollary~\ref{Rn}, we see that $E(\Gamma_k\ltimes\R^2)<_B E({\Gamma}^z\ltimes\R^2)$.

Assume toward a contradiction that there is a topological isomorphism $\phi:G_0\to H_0$ such that $\phi{\rm Inn}_G(G_0)\phi^{-1}\subseteq{\rm Inn}_H(H_0)$. Note that $\phi{\rm Inn}_G(G_0)\phi^{-1}$ is a finite group. However, any nontrivial subgroup of ${\rm Inn}_H(H_0)$ is infinite. This is a contradiction!
\end{proof}

\subsection*{Acknowledgements}
The author would like to Longyun Ding for carefully reading and suggestions on proofs of the article.

\end{document}